\documentclass[a4paper,11pt]{amsart}
\usepackage[utf8]{inputenc}          
\usepackage[T1]{fontenc}
\usepackage{amsmath,amssymb}
\usepackage{tikz-cd}                 
\usepackage[all]{xy}                 
\usepackage{xcolor}                  
\definecolor{Periwinkle}{rgb}{0.8,0.8,1.0}
\definecolor{Emerald}{rgb}{0.0,0.8,0.4}
\usepackage[colorlinks, linkcolor=blue, anchorcolor=Periwinkle,
    citecolor=blue, urlcolor=Emerald]{hyperref}
\usepackage[mathscr]{euscript}

\SelectTips{cm}{}

\DeclareUnicodeCharacter{2212}{-}

\newcommand{\ie}{{\em i.e.}\ }

\newcommand{\ko}{\: , \;}

\numberwithin{equation}{subsection}
\newtheorem{theorem}[subsection]{Theorem}
\newtheorem{classification-theorem}[subsection]{Classification Theorem}
\newtheorem{decomposition-theorem}[subsection]{Decomposition Theorem}
\newtheorem{definition}[subsection]{Definition}
\newtheorem{periodicity-conjecture}[subsection]{Periodicity Conjecture}

\newtheorem{proposition}[subsection]{Proposition}
\newtheorem{corollary}[subsection]{Corollary}

\newtheorem{example}[subsection]{Example}
\newtheorem{remark}[subsection]{Remark}

\newcommand{\reminder}[1]{}

\newcommand{\opname}[1]{\operatorname{\mathsf{#1}}}

\renewcommand{\mod}{\opname{mod}\nolimits}

\newcommand{\proj}{\opname{proj}\nolimits}

\newcommand{\per}{\opname{per}\nolimits}

\renewcommand{\Im}{\opname{Im}\nolimits}

\newcommand{\Z}{\mathbb{Z}}

\newcommand{\Ker}{\opname{Ker}}

\newcommand{\Hom}{\opname{Hom}}
\newcommand{\sg}{\opname{sg}}

\newcommand{\rad}{\opname{rad}}

\newcommand{\ten}{\otimes}

\newcommand{\Tor}{\opname{Tor}}

\newcommand{\ca}{{\mathscr A}}
\newcommand{\cb}{{\mathscr B}}
\newcommand{\cc}{{\mathscr C}}
\newcommand{\cd}{{\mathscr D}}

\newcommand{\cs}{{\mathscr S}}

\renewcommand{\phi}{\varphi}

\begin{document}

\date{June 7, 2026} 

\title[Singular Hochschild complex and Cartan matrix]{Singular Hochschild complex and Cartan matrix}
\author{Yu Wang}
\address{School of Mathematics and Statistics, Taiyuan Normal University, Shanxi 030619, P. R. China}
\email{yu.wang@tynu.edu.cn}

\author{Xiaozhuan Liang}
\address{School of Mathematics and Statistics, Taiyuan Normal University, Shanxi 030619, P. R. China}
\email{3518143443@qq.com}

\begin{abstract}
If $A$ is a symmetric algebra, then Hochschild homology
of the dg enhancement of the singularity category of $A$ agrees with singular Hochschild
homology of $A$. For a basic finite‑dimensional $k$-algebra $A$, the Cartan matrix of $A$ is symmetric if and only if the $k$-dual of the mixed complex of the dg enhancement of its singularity category is isomorphic to its shift by −1. We provide two counterexamples to show that neither result holds for general Frobenius algebras. 
\end{abstract}


\maketitle

\section{Introduction}

Let $k$ be a commutative ring and $A$ be a
right Noetherian (non-commutative) $k$-algebra projective over $k$. We write $\ten$ for $\ten_k$. 
The {\em stable derived category} or {\em singularity category} of $A$ is
defined as the Verdier quotient 
\[
\sg(A)=\cd^b(\mod A)/\per(A)
\]
of the bounded derived category of finitely generated (right) $A$-modules by
the {\em perfect derived category $\per(A)$}, 
\ie the full subcategory of complexes quasi-isomorphic
to bounded complexes of finitely generated projective modules. It was introduced
by Buchweitz in an unpublished manuscript \cite{Buchweitz86} in 1986
and rediscovered, in its scheme-theoretic variant, by Orlov in 2003 \cite{Orlov04}.
This category vanishes when $A$ has finite global dimension.

Suppose that the enveloping algebra $A^e=A\ten A^{op}$ is also right
Noetherian. Hochschild homology $HH_{n}(A)$ of the algebra $A$ can be defined as the homology of the Hochschild chain complex or, equivalently, by 
\[
HH_{n}(A) \cong \Tor^{A^e}_n(A,A)\ko n\in\Z .
\]
In particular, $HH_0(A)\cong A/[A,A]$ is the quotient by commutators.  Hochschild homology is invariant under Morita and derived equivalences. One similarly defines Hochschild homology $HH_{*}(\mathcal{C})$ for any dg category $\mathcal{C}$ \cite{Keller98}, in particular for the canonical dg enhancement $\cs$ of the singularity category of $A$ \cite{ChenWang23,WAK}. It is a dg quotient  category \cite{Drinfeld04, Keller99}. 

For a Gorenstein algebra $A$, one considers the singular Hochschild homology $HH_{*}^{sg}(A)$, defined (for example) using a complete (Tate) resolution of $A$ over $A^e$ \cite{BerghJorgensen13}. If the Gorenstein dimension of the enveloping algebra is $d$, then for every $n\geq d+1$
there are isomorphisms $HH_{n}^{sg}(A)\cong HH_{n}(A)$. In particular, if $A$ is a Frobenius algebra, then
there are isomorphisms $HH_{n}^{sg}(A)\cong HH_{n}(A)$ for every $n\geq 1$. Moreover, there are isomorphisms \cite{BerghJorgensen13, EuSchedler09}
\[HH_{n}^{sg}(A)\cong HH_{-(n+1)}^{sg}(A)\ko n\in\Z .\]

In this paper, our goal is to compare $HH_{*}^{sg}(A)$ and $HH_{*}(\cs)$. Our main result is the following.
\begin{theorem}
 Let $k$ be an algebraically closed field and $A$ a finite-dimensional basic $k$-algebra. The Cartan matrix of $A$ is symmetric if and only if $DM_*(\cs)\cong\Sigma ^{-1} M_*(\cs)$.  
\end{theorem}
The complex $M_*(\cs)$ is  introduced by Wang, Arunachalam and Keller \cite{WAK} to study $HH_{*}(\cs)$ and $D$ is the $k$-dual.
\begin{corollary}
If $A$ is symmetric, then we have a canonical isomorphism
\[HH_{*}^{sg}(A)\cong HH_{*}(\cs).\]
\end{corollary} 
There is an exact sequence:
\[
\xymatrix{
0\ar[r] &HH_{1}(\cs) \ar[r] &HH_{0}(A)\ar[r] &DHH_{0}(A) \ar[r] & HH_{0}(\cs) \ar[r] & 0.}
\]
In the paper \cite{WAK}, they defined a natural trace pairing on $A$ by
\[
\langle a,b\rangle = Tr(\lambda_a\circ \rho_b),
\]
where $\lambda_a,\rho_b$ denote left- and right-multiplication by $a,b\in A$.  We prove that this pairing is symmetric if and only if the Cartan matrix of $A$ is symmetric.  As a consequence, for basic algebras with symmetric Cartan matrix, the mixed complex computing $HH_*(\cs)$ is self-dual up to a shift. We illustrate these results with explicit examples, including computations of trace pairings for some finite dimensional algebras. We give two examples of Frobenius algebras for which the isomorphisms $DM_*(\cs)\cong\Sigma ^{-1} M_*(\cs)$ and $HH_{*}^{sg}(A)\cong HH_{*}(\cs)$ do not hold. 

In section \ref{basic}, we recall basic definitions and properties of Hochschild homology and mixed complexes. In section \ref{s:main}, we prove our main results and present the counterexamples.

\section{Hochschild Homology and Mixed Complexes}
\label{basic}

Let $\cd A=\cd(\mod A)$ denote the unbounded derived category of finitely generated (right) $A$-modules.

Hochschild homology $HH_{*}(A)$ of $A$ is the homology of the Hochschild complex $C_*(A)$: 
\[
\xymatrix{
A  & A\otimes A \ar[l] &\cdots\ar[l] & A^{\otimes p} \ar[l] & A^{\otimes p+1} \ar[l] &\cdots\ar[l]}
\]
with $C_{p}(A)=A^{\otimes(p+1)}$, $p\ge0$, and differential given by
\[
d(a_{0},\ldots, a_{p})=\sum_{i=0}^{p-1}(-1)^{i}(a_{0},\ldots, a_{i}a_{i+1},\ldots, a_{p})+(-1)^p(a_pa_0,\ldots, a_{p-1}),
\]
where we write $(a_{0},\ldots, a_{p})$ for $a_{0}\otimes\ldots\otimes a_{p}$. Notice that the first differential takes $a\otimes b$ to the commutator $ab-ba$.
\begin{remark}
We see $HH_{0}(A)=A/[A, A]$, where $[A, A]$ is a $k$-submodule of $A$ and is generated by $ab-ba$ for any $a, b\in A$. If $[A, A]$ is not an ideal of the algebra $A$, then $HH_{0}(A)$ is not an algebra. If $HH_{0}(A)$ is an algebra, then it is commutative.
\end{remark}
The definition of the Hochschild complex can be extended from a $k$-algebra $A$ to a small $k$-category $\ca$.

\[
\begin{tikzcd}[column sep=small]
{\coprod\limits_{A_{0}\in\ca}\ca(A_{0}, A_{0})} & {\coprod\limits_{A_{0}, A_{1}\in\ca}\ca(A_{1}, A_{0})\otimes \ca(A_{0}, A_{1})} & \cdots \\
{} & {\coprod\limits_{A_{0},\ldots, A_{p}\in\ca}\ca(A_{p}, A_{0})\otimes\ca(A_{p-1}, A_{p})\otimes\ldots\otimes\ca(A_{0}, A_{1})} & \cdots
\arrow[from=1-2, to=1-1]
\arrow[from=1-3, to=1-2]
\arrow[from=2-2, to=2-1]
\arrow[from=2-3, to=2-2]
\end{tikzcd}
\]

with $C_{p}(\ca)=\coprod\limits_{A_{0},...,A_{p}\in\ca}\ca(A_{p}, A_{0})\otimes\ca(A_{p-1}, A_{p})\otimes\ldots\otimes\ca(A_{0}, A_{1})$, $p\ge0$, and differential given by
\[
d(a_{0},\ldots, a_{p})=\sum_{i=0}^{p-1}(-1)^{i}(a_{0},\ldots, a_{i}a_{i+1},\ldots, a_{p})+(-1)^p(a_pa_0,\ldots, a_{p-1}),
\]
where $a_{0}\in\ca(A_{p}, A_{0})$, $a_{i}\in\ca(A_{p-i}, A_{p-i+1})$ for all $p-i\ge 0$, and $(a_{0},\ldots,a_{p})$ denotes $a_{0}\otimes\ldots\otimes a_{p}$.

One has $HH_{*}(A)\cong HH_{*}(\proj A)$ in $\cd k$. This yields Morita invariance of Hochschild homology. The definitions further extend to small differential graded (dg) $k$-categories $\ca$. For example, $HH_{*}(A)\cong HH_{*}(\proj A)\cong HH_{*}(\cc^{b}_{dg}(\proj A))$ in $\cd k$, where $\cc^{b}_{dg}(\proj A)$ is the dg category of bounded projective complexes.
This yields derived Morita invariance of Hochschild homology.

\begin{theorem}\cite{Keller98}
If $\xymatrix{
0\ar[r] &\ca\ar[r] & \cb \ar[r] & \cc \ar[r] & 0}$ is a sequence of dg categories such that the derived sequence
\[
\xymatrix{
0\ar[r] &\cd\ca\ar[r] & \cd\cb \ar[r] & \cd\cc \ar[r] & 0}
\]
is exact, then there is a canonical triangle
\[
\xymatrix{
C(\ca)\ar[r] & C(\cb) \ar[r] & C(\cc) \ar[r] & \Sigma C(\ca) }
\] 
in the mixed derived category.
\end{theorem}
We give a brief introduction to mixed complexes.

\begin{definition}\label{def:mixed}
A mixed complex $C$ is given by:
\[
\begin{tikzcd}[column sep=10mm]
\cdots \arrow["b", shift left=1, r] & C^{-1} \arrow["B", shift left=1, l] \arrow["b", shift left=1, r] & C^{0} \arrow["B", shift left=1, l] \arrow["b", shift left=1, r] & C^{1} \arrow["B", shift left=1, l] \arrow["b", shift left=1, r] & \cdots \arrow["B", shift left=1, l]
\end{tikzcd}
\]
such that the following conditions hold:
\[
b^{2}=0,\quad B^{2}=0,\quad bB+Bb=0.
\]
\end{definition}

\begin{remark}
A mixed complex is both a chain complex and a cochain complex in a compatible way. Equivalently, it is a dg module over the dg algebra $\Lambda=k[\epsilon]/(\epsilon)^{2}$, with $|\epsilon|=-1$ and $d=0$.
\end{remark}

\begin{example}\cite{Loday98}
If $A$ is an algebra, then we have the Connes-Quillen cyclic bicomplex:
\[
\begin{tikzcd}[column sep=10mm, row sep=7mm]
{} & {} & {} & {} & {} \\
A^{\otimes3} & A^{\otimes3} & A^{\otimes3} & A^{\otimes3} & \cdots \\
A^{\otimes2} & A^{\otimes2} & A^{\otimes2} & A^{\otimes2} & \cdots \\
A & A & A & A & \cdots
\arrow["0"', from=4-2, to=4-1]
\arrow["1"', from=4-3, to=4-2]
\arrow["d", from=3-1, to=4-1]
\arrow["{b'}", from=3-2, to=4-2]
\arrow["d", from=3-3, to=4-3]
\arrow["{1-t}"', from=3-2, to=3-1]
\arrow["{1+t}"', from=3-3, to=3-2]
\arrow["d", from=2-1, to=3-1]
\arrow["{b'}", from=2-2, to=3-2]
\arrow["d", from=2-3, to=3-3]
\arrow["0"', from=4-4, to=4-3]
\arrow["{1-t}"', from=3-4, to=3-3]
\arrow["{1-t}"', from=2-4, to=2-3]
\arrow["{1+t+t^{2}}"', from=2-3, to=2-2]
\arrow["{1-t}"', from=2-2, to=2-1]
\arrow["d", from=1-1, to=2-1]
\arrow["{b'}", from=1-2, to=2-2]
\arrow["d", from=1-3, to=2-3]
\arrow["1"', from=4-5, to=4-4]
\arrow["{1+t}"', from=3-5, to=3-4]
\arrow["{1+t+t^2}"', from=2-5, to=2-4]
\arrow["{b'}", from=1-4, to=2-4]
\arrow["{b'}", from=2-4, to=3-4]
\arrow["{b'}", from=3-4, to=4-4]
\end{tikzcd}
\]
Here $d$ is the differential of the Hochschild complex, $b'$ is the differential of the bar complex, and the cyclic group $\Z/(n+1)\Z$ action on $A^{\otimes n+1}$ is given by the generator $t=t_{n}$:
\[
t_{n}(a_{0}\otimes\ldots\otimes a_{n})=(-1)^{n}a_{n}\otimes a_{0}\otimes\ldots\otimes a_{n-1}.
\]
The homology of the total complex is the cyclic homology of $A$. Let $MA$ denote the total complex (cone of $(1-t)$) over the subcomplex formed by the first two columns.  
$b$ is the mapping cone differential $\left(\begin{smallmatrix} d & 1-t \\ 0 & -b'\end{smallmatrix}\right)$ and $B$ is $\left(\begin{smallmatrix} 0 &0 \\ N & 0\end{smallmatrix}\right): MA\to MA$ with $N=1+t+\cdots+t^{n}$, the homogeneous map of (cohomology) degree $-1$.
We have $H^n(MA)\cong HH_{-n}(A)$ for all $n\in\Z$.
\end{example}

\section{Singular Hochschild Homology}
\label{s:main}

There is a useful fact about Frobenius (symmetric) algebras. 
\begin{proposition}\cite[Corollary 3.3]{BerghJorgensen13} 
If $A$ is a Frobenius algebra, then the enveloping algebra $A^e$ is also Frobenius. Moreover, if $A$ is symmetric, then $A^e$ is symmetric.
\end{proposition}

\begin{proof}
Since $A$ is a Frobenius algebra, let $\nu: A \to A$ be its Nakayama automorphism, so that we have the following bimodule isomorphisms:
\[
\begin{tikzcd}
D(_{\nu}A) \arrow[r, "\sim"] & A_\nu \arrow[r, "\nu^{-1}"] & {}_{\nu^{-1}}A.
\end{tikzcd}
\]
Here, the bimodule structures are given by:
\begin{itemize}
\item For $D(A) = \mathrm{Hom}_k(A,k)$: $(a \cdot f \cdot b)(x) = f(bxa)$ for $f \in D(A)$, $a,b,x \in A$.
\item For $A_\nu$: $a \cdot x \cdot b = ax\nu(b)$.
\item For ${}_{\nu^{-1}}A$: $a \cdot x \cdot b = \nu^{-1}(a)xb$.
\end{itemize}

For any idempotent $e \in A$, we have:
\[
D(_{\nu}Ae) \cong eD(_{\nu}A) \cong eA_\nu, \quad D(eA_{\nu}) \cong D(A_{\nu})e \cong {}_{\nu^{-1}}Ae.
\]

More generally, for a projective left $A$-module $P$, define $P^*= \mathrm{Hom}_A(P,A)$ (a right $A$-module). Then:
\[
D(P) \cong (P^*)_\nu,
\]
where the right action on $(P^*)_\nu$ is twisted by $\nu$.

For a projective right $A$-module $Q$, define $Q^* = \mathrm{Hom}_A(Q,A)$ (a left $A$-module). Then:
\[
D(Q) \cong {}_{\nu^{-1}}(Q^*)
\]
where the left action on ${}_{\nu^{-1}}(Q^*)$ is twisted by $\nu^{-1}$.

There are the following isomorphisms:
\[
D(P \otimes Q) \cong D(Q) \otimes D(P) \cong {}_{\nu^{-1}}(Q^*) \otimes (P^*)_\nu.
\]

In particular, take $P = {}_A A$ and $Q = A_A$. Then $P \otimes Q \cong A^e$, and we have:
\[
P^* = \mathrm{Hom}_A({}_A A, A) \cong A_A, \quad Q^* = \mathrm{Hom}_A(A_A, A) \cong {}_A A.
\]
Therefore:
\[
D(A^e) \cong {}_{\nu^{-1}}A \otimes A_\nu.
\]
So $A^e$ is a Frobenius algebra with $\nu\otimes\nu^{-1}$ as its Nakayama automorphism. Note that $A$ is symmetric if and only if its Nakayama automorphism is the identity. Hence, if $A$ is symmetric, then $A^e$ is symmetric.
\end{proof}

Our goal is to compare singular Hochschild homology $HH^{sg}_{*}(A)$ and $HH_{*}(\cs)$ for a finite dimensional algebra. Let $A$ be a Frobenius algebra, so $A^e$ is also Frobenius. To compute $HH^{sg}_{*}(A)$, we choose a projective resolution $P_*$ of $A$ over $A^{e}$:
\[
\xymatrix{
\cdots\ar[r] & P_{2} \ar[r] & P_{1} \ar[r] & P_{0} \ar[r] & A\ar[r] & 0}.
\] 
Consider its dual complex:
\[
\begin{tikzcd}
	0 & DA & {DP_{0}} & {DP_{1}} & \cdots \\
	0 & {_{\nu^{-1}}A_{\nu}} & {_{\nu^{-1}}P^{\vee}_{0\nu}} & {_{\nu^{-1}}P^{\vee}_{1\nu}} & \cdots
	\arrow[from=1-1, to=1-2]
	\arrow[from=1-2, to=1-3]
	\arrow[from=1-3, to=1-4]
	\arrow[from=1-4, to=1-5]
	\arrow[from=2-1, to=2-2]
	\arrow[from=2-2, to=2-3]
	\arrow[from=2-3, to=2-4]
	\arrow[from=2-4, to=2-5]
	\arrow["\sim", from=1-2, to=2-2]
	\arrow["\sim", from=1-3, to=2-3]
	\arrow["\sim", from=1-4, to=2-4],
\end{tikzcd}
\]
where $P_*^{\vee}=\Hom_{A^{e}}(P_*, A^{e})$. Twist by $\nu^{-1}$ on the right:
\[
\xymatrix{
0 \ar[r]&_{\nu^{-1}}A \ar[r] & _{\nu^{-1}}P^{\vee}_{0} \ar[r] & _{\nu^{-1}}P^{\vee}_{1} \ar[r] & _{\nu^{-1}}P^{\vee}_{2} \ar[r] & \cdots}.
\] 

Whence there is a complete resolution (denoted by $CR$):
\[
\begin{tikzcd}
	\cdots & {P_{1}} & {P_{0}} && {_{\nu^{-1}}P^{\vee}_{0}} & {_{\nu^{-1}}P^{\vee}_{1}} & \cdots \\
	&&& A
	\arrow[from=1-2, to=1-3]
	\arrow[from=1-3, to=1-5]
	\arrow[from=1-3, to=2-4]
	\arrow[from=2-4, to=1-5]
	\arrow[from=1-1, to=1-2]
	\arrow[from=1-5, to=1-6]
	\arrow[from=1-6, to=1-7]
\end{tikzcd}
\]
In this complex, the term ${}_{\nu^{-1}}P^{\vee}_{0}$ sits in degree $0$; consequently $P_0$ is in degree $1$, $P_1$ in degree $2$, and in general $P_n$ is in degree $n+1$. Thus for $p\geq 2$ we have $CR_p = P_{p-1}$, and the truncation $(A\otimes_{A^e}CR)_{\ge 2}$ is isomorphic to $(A\otimes_{A^e}P_*)_{\ge 1}$ with a degree shift by one. This yields an unbounded chain complex with $_{\nu^{-1}}P^{\vee}_{0}$ in degree $0$. We are interested in $HH^{sg}_{p}(A):=H_{p}(A\otimes_{A^{e}}CR)$. The complex $A\otimes_{A^e}CR$ is depicted as:
\[
\begin{tikzcd}
	\cdots & {A\otimes_{A^{e}} P_{1}} & {A\otimes_{A^{e}} P_{0}} & {A\otimes_{A^{e}} {_{\nu^{-1}}}P^{\vee}_{0}} & {A\otimes_{A^{e}} {_{\nu^{-1}}}P^{\vee}_{1}} & \cdots
	\arrow[from=1-2, to=1-3]
	\arrow[from=1-3, to=1-4]
	\arrow[from=1-1, to=1-2]
	\arrow[from=1-4, to=1-5]
	\arrow[from=1-5, to=1-6]
\end{tikzcd}
\]
Now we compare $HH^{sg}_*(A)$ with $HH_*(\cs)$. Recall from \cite{WAK} (see also Section~\ref{basic}) that there is a canonical triangle in the mixed derived category
\[
\xymatrix{
MA \ar[r] & DMA \ar[r] & M(\cs) \ar[r] & \Sigma MA,}
\]
which yields a long exact sequence relating $HH_*(A)$, its dual, and $HH_*(\cs)$. In particular, for all $p\notin\{0,1\}$ the maps in this sequence give isomorphisms
\[
HH_{p-1}(A) \;\cong\; HH_{p}(\cs) \qquad (\text{for } p\ge 2),
\]
\[
D HH_{-p}(A) \;\cong\; HH_{p}(\cs) \qquad (\text{for } p\le -1).
\]

\textbf{Case $p\ge 2$.} Using the degree-shift observation above we obtain:
\[
\begin{tikzcd}
HH^{sg}_{p}(A) \ar[r,"\sim"] & H_{p}(A\otimes_{A^{e}} CR) \ar[r,"\sim"] & H_{p-1}(A\otimes_{A^{e}} P_{*}) \ar[d,equal] \\
& & HH_{p-1}(A) \ar[r,"\sim"] & HH_{p}(\cs)
\end{tikzcd}
\]
The last isomorphism follows from the aforementioned triangle.

\textbf{Case $p\le -1$.} Here $CR_p = {}_{\nu^{-1}}P^{\vee}_{p}$, and we compute:
\begin{align*}
A\otimes_{A^{e}}CR_{p}
&\cong A\otimes_{A^{e}} {_{\nu^{-1}}}P^{\vee}_{p} \\
&\cong A_{\nu}\otimes_{A^{e}}P^{\vee}_{p} \\
&\cong \Hom_{A^{e}}(P_{p}, A_{\nu}) \\
&\cong \Hom_{A^{e}}(P_{p}, DA).
\end{align*}
Passing to homology and using that $P_*$ is a projective resolution, we obtain:
\begin{align*}
HH^{sg}_{p}(A)
&\cong H_{p}\mathbb{R}\!\Hom_{A^{e}}(A, DA) \\
&\cong H_{p}\,D(A\otimes^{\mathbb{L}}_{A^{e}} A) \\
&\cong D\,HH_{-p}(A,A) \\
&\cong HH_{p}(\cs),
\end{align*}
where the final isomorphism comes again from the mixed triangle.

Hence, for all $p\notin\{0,1\}$ we have proved the isomorphism
\[
HH^{sg}_{p}(A) \;\cong\; HH_{p}(\cs).
\]
What happens for $p=0$ and $p=1$? The long exact sequence from the triangle $MA\to DMA\to M(\cs)\to\Sigma MA$ yields the exact sequence
\[
\xymatrix{
0\ar[r] &H_{1}M(\cs) \ar[r] &H_{1}\Sigma MA\ar[r] &H_{0}DMA \ar[r] & H_{0}M(\cs) \ar[r] & 0.}
\]
Equivalently,
\[
\xymatrix{
0\ar[r] &HH_{1}(\cs) \ar[r] &HH_{0}(A)\ar[r] &DHH_{0}(A) \ar[r] & HH_{0}(\cs) \ar[r] & 0.}
\]
This completes the comparison at the remaining two degrees.

There exists a complex of $A$-bimodules which connects the Hochschild complex and its $k$-dual via a $k$-linear map $\tau$:  
\[
\begin{tikzcd}[column sep=6mm]
	\cdots & {A\ten A} & {A} & {} & {DA} & {D(A\ten A)} & \cdots \\
	&&& {A/[A,A]\to D(A/[A,A])} 
	&&& 
	\arrow["d", from=1-2, to=1-3]
	\arrow["Dd", from=1-5, to=1-6]
	\arrow[from=1-1, to=1-2]
	\arrow[from=1-6, to=1-7]
	\arrow["\tau", from=1-3, to=1-5]
	\arrow[from=1-3, to=2-4]
	\arrow[from=2-4, to=1-5]
\end{tikzcd}
\]
where $d$ is the commutator ($A\otimes A\to A$, $a\otimes b\mapsto ab-ba$) and the degree of $A$ is $-1$. For any $a,b\in A$,
\[
(\tau(a))(b):=Tr(\lambda_{a}\circ\rho_{b}: A\to A).
\]
Notice that $\tau d=0$ and $Dd \tau=0$, because $Tr(\lambda_{a}\circ\rho_{b})$ only depends on $[a]$, $[b]$. Indeed:
\begin{align*}
Tr(\lambda_{[a, a']}\circ\rho_{b})
&= Tr((\lambda_{a}\lambda_{a'}-\lambda_{a'}\lambda_{a})\circ\rho_{b})\\
&= Tr(\lambda_{a}\lambda_{a'}\rho_{b}-\lambda_{a'}\rho_{b}\lambda_{a})\\
&= 0.
\end{align*}
Note that $\rho_{x}\lambda_{y}=\lambda_{y}\rho_{x}$ for all $x, y\in A$. Similarly, we have $Tr(\lambda_{a}\circ\rho_{[b, b']})=0$ for all $a, b, b'\in A$. There is the following theorem:

\begin{theorem}\cite{WAK} $HH_*(\cs)$ is canonically isomorphic (in $\cd k$) to the homology of the double Hochschild complex of $A$:
\[
\xymatrix{
\cdots\ar[r] & A\ten A \ar[r]^{d} & A \ar[r]^{\tau}& DA \ar[r] & D(A\ten A)\ar[r] & \cdots}.
\] 
where $(\tau(a))(b)=Tr(\lambda_{a}\circ\rho_{b})$, with the left multiplication $\lambda_{a}$ and the right multiplication $\rho_{b}$.
\end{theorem}

\begin{remark}The theorem can be generalized to suitable proper dg algebras.
\end{remark}

We denote the double Hochschild complex by $M_*(\cs)$. In general, does $DM_*(\cs)\cong\Sigma ^{-1} M_*(\cs)$ hold for an arbitrary finite-dimensional algebra?
The answer is no. We need to find the conditions for this isomorphism to hold. 

\begin{proposition}\label{complex}
 The double Hochschild complex $M_*(\cs)$ is isomorphic to its $k$-dual shift complex $\Sigma D(M_*(\cs))$ if and only if    
$Tr(\lambda_{a}\circ\rho_{b})=Tr(\lambda_{b}\circ\rho_{a})$ for all $a, b\in A$, where $\lambda_{a}\circ\rho_{b}: x\mapsto axb$ and $\lambda_{b}\circ\rho_{a}: x\mapsto bxa$.
\end{proposition}

\begin{proof}
Suppose that there is the following commutative diagram:
\[
\xymatrix@C=1.2cm{
\cdots \ar[r] & A\ten A \ar[r]^{d} \ar[d] & A \ar[r]^{\tau} \ar[d]^{\phi} & DA \ar[r] \ar[d]^{=} & D(A\ten A) \ar[r] \ar[d]^{=} & \cdots \\
\cdots \ar[r] & D(D(A\ten A)) \ar[r]^{D(Dd)} & D(DA) \ar[r]^{D\tau} & DA \ar[r] & D(A\ten A) \ar[r] & \cdots,
}
\]
where $\phi$ is the canonical isomorphism from a finite dimensional algebra $A$ to its double $k$-dual $D(DA)$.  We only need to consider the commutativity of the second square (the other squares commute naturally). We have:
\[
(\tau(a))(b)=Tr(\lambda_a\circ\rho_b),\qquad
D\tau(\phi(a))(b)=\phi(a)(\tau(b))=\tau(b)(a)=Tr(\lambda_b\circ\rho_a).
\]
The diagram commutes if and only if $Tr(\lambda_a\circ\rho_b)=Tr(\lambda_b\circ\rho_a)$ for all $a,b\in A$.
\end{proof}

We define a bilinear form $\langle a, b\rangle:=Tr(\lambda_{a}\circ\rho_{b})$. Hence, the isomorphism $DM_*(\cs)\cong\Sigma ^{-1} M_*(\cs)$ holds if and only if the $k$-bilinear form is symmetric ($\langle a, b\rangle=\langle b, a\rangle$ for all $a, b\in A$). Let $\{e_1,\dots,e_n\}$ be a basis of $A$. For finite dimensional algebras, we calculate the matrix $(p_{ij})\in M_n(k)$ of $\tau$, defined by $p_{ij}:=\langle e_i, e_j\rangle$ for $i,j\in \{1,\dots,n\}$.  

\begin{example}
Let $Q: 1\to 2$ be the $A_2$ quiver, then its path algebra $kQ$ is isomorphic to the matrix algebra $\left( \begin{matrix} k &0 \\ k& k\end{matrix}  \right)$. Suppose that its basis is $\{E_{11},E_{21},E_{22}\}$. The products of these basis elements are:
\begin{align*}
E_{11}E_{11}&=E_{11} &E_{11}E_{21}&=0 &E_{11}E_{22}&=0,\\
E_{21}E_{11}&=E_{21} &E_{21}E_{21}&=0 &E_{21}E_{22}&=0,\\
E_{22}E_{11}&=0 &E_{22}E_{21}&=E_{21} &E_{22}E_{22}&=E_{22}.\\
\end{align*}
The matrices of left and right multiplication are:
\begin{align*}
\lambda_{E_{11}}&=\left( \begin{matrix} 1 &0&0 \\ 0& 0& 0\\0& 0& 0\end{matrix}  \right) 
&\lambda_{E_{21}}&=\left( \begin{matrix} 0&0&0 \\ 1&0& 0\\0& 0& 0\end{matrix}  \right)
&\lambda_{E_{22}}&=\left( \begin{matrix} 0&0&0 \\ 0&1& 0\\0& 0& 1\end{matrix}  \right),\\
\rho_{E_{11}}&=\left( \begin{matrix} 1&0&0 \\ 0&1&0\\0&0& 0\end{matrix}  \right)
&\rho_{E_{21}}&=\left( \begin{matrix} 0&0&0 \\ 0&0&1\\0&0& 0\end{matrix}  \right)
&\rho_{E_{22}}&=\left( \begin{matrix} 0&0&0 \\ 0&0&0\\0&0&1\end{matrix}  \right).\\
\end{align*}
We compute $\langle a, b\rangle=Tr(\lambda_{a}\circ\rho_{b})$:
\begin{align*}
\langle E_{11}, E_{11}\rangle &=Tr \left( \begin{matrix} 1&0&0 \\ 0&0&0\\0&0&0\end{matrix}  \right)=1, &
\langle E_{11}, E_{21}\rangle &=Tr(0)=0,\\
\langle E_{11}, E_{22}\rangle &=Tr(0)=0, &
\langle E_{21}, E_{11}\rangle &=Tr(\lambda_{E_{21}})=0,\\
\langle E_{21}, E_{21}\rangle &=Tr(0)=0, &
\langle E_{21}, E_{22}\rangle &=Tr(\lambda_{E_{21}})=0,\\
\langle E_{22}, E_{11}\rangle &=Tr\left( \begin{matrix} 0&0&0 \\ 0&1&0\\0&0&0\end{matrix}  \right)=1, &
\langle E_{22}, E_{21}\rangle &=Tr(\rho_{E_{21}})=0,\\
\langle E_{22}, E_{22}\rangle &=Tr\left( \begin{matrix} 0&0&0 \\ 0&0&0\\0&0&1\end{matrix}  \right)=1.
\end{align*}
Hence, the matrix of $\langle\cdot,\cdot\rangle$ is $\left( \begin{matrix} 1&0&0 \\ 0&0&0\\1&0&1\end{matrix}  \right)$, so the form is not symmetric for $Q: 1\to 2$.
\end{example}

\begin{example}\label{3.6}
Let $Q$ be the quiver 
\[
\begin{tikzcd}[column sep=small]
	1 & 2
	\arrow["\alpha", shift left=2, from=1-1, to=1-2]
	\arrow["\beta", shift left=2, from=1-2, to=1-1]
\end{tikzcd}
\]
and $I$ be the ideal of $kQ$ generated by $\alpha\beta$ and $\beta\alpha$. The matrix of $\langle\cdot,\cdot\rangle$ is 
\[
\begin{pmatrix}
1&1&0&0 \\
1&1&0&0 \\
0&0&0&0 \\
0&0&0&0
\end{pmatrix}
\]
with respect to the basis $\{e_1, e_2, \alpha, \beta \}$. The non‑zero $2\times 2$ block in the upper left corner, $\left( \begin{matrix} 1&1 \\ 1&1\end{matrix} \right)$, is the Cartan matrix of $A$. 
One finds that $\Ker d=\Im \tau=\operatorname{span}\{\alpha, \beta\}$. Hence, $HH_0(\cs)=0$. Define a Frobenius trace $t: A\to k$ by 
\[
t(e_1)=0,\quad t(e_2)=0,\quad t(\alpha)=1,\quad t(\beta)=1,
\]
and the bilinear form $\phi(x,y)=t(xy)$.  One checks that $\phi$ is nondegenerate but not symmetric.

On the other hand, the Nakayama automorphism $\nu$ is determined by 
\[
\phi(a,b)=\phi(b,\nu(a))\quad\forall a,b\in A.
\]
By evaluating on the basis one finds 
\[
\nu(e_1)=e_2,\quad \nu(e_2)=e_1,\quad 
\nu(\alpha)=\beta,\quad \nu(\beta)=\alpha,
\]
so $\nu$ has order two. Consider two projective $A^e$-modules:
\[
P_0 = Ae_1\otimes e_1A\oplus Ae_2\otimes e_2A,\quad
P_1 = Ae_1\otimes e_2A\oplus Ae_2\otimes e_1A.
\]
Consider the complete resolution
\[
\xymatrix{
& \cdots \ar[r] & P_0 \ar[r]^{d_0} & P_1 \ar[r]^{d_1} & P_0 \ar[rr]^{d_0} \ar[dr]^{\mu} & & P_1 \ar[r]^{d_1} & P_0 \ar[r] & \cdots \\
& & & & & A \ar[ur] & & & & &
}
\]
where $P_1\cong {}_{\nu^{-1}}P^{\vee}_0$, $P_0\cong {}_{\nu^{-1}}P^{\vee}_1$ as $A^e$-modules, and the differentials are as follows:
\[
\begin{aligned}
d_0(e_1\otimes e_1) &= e_1\otimes\beta-\alpha\otimes e_1,\quad
d_0(e_2\otimes e_2) = e_2\otimes\alpha-\beta\otimes e_2,\\
d_1(e_1\otimes e_2) &= e_1\otimes\alpha-\alpha\otimes e_2,\quad
d_1(e_2\otimes e_1) = e_2\otimes\beta-\beta\otimes e_1.
\end{aligned}
\]
And $\mu: P_0 \to A$ is the multiplication map  
\[
\mu(ae_i\otimes e_i b) = a\,e_i b = a b \quad(a,b\in A).
\]

Since $A$ is a Frobenius algebra, 
there are isomorphisms $HH_{n}^{sg}(A)\cong HH_{n}(A)$ for every $n\geq 1$ \cite{BerghJorgensen13, EuSchedler09}.
Therefore, from the 2-periodic complete resolution we obtain 
\[
HH_{n}^{sg}(A)\cong H_n(P_\bullet\otimes_{A^e}A) \cong k,\quad n\in \Z.
\]
\end{example}

\begin{proposition}\label{Cartan}
Let $A=kQ/I$, where $Q$ is a finite quiver, $k$ a field, and $I$ is an admissible ideal of $kQ$. If we choose a basis of $A$ with the trivial paths $\{e_1,\dots,e_n\}$ as the first $n$ elements, followed by non-trivial paths, then the matrix of the bilinear form $\langle\cdot,\cdot\rangle$ with respect to this basis has its only non-zero entries in the upper-left $n \times n$ block, and that block is precisely the Cartan matrix $C$ of $A$.
\end{proposition}
\begin{proof}
For $a,b\in A$ define $\alpha_{a,b}:A\to A$ by $\alpha_{a,b}(x)=axb$. Then $\langle a,b\rangle = Tr(\alpha_{a,b})$. Therefore, $\alpha^{2}_{a, b}: x\mapsto aaxbb = \lambda_{a^{2}}x\rho_{b^{2}}$, i.e.\ $\alpha^{2}_{a, b}=\alpha_{a^{2}, b^{2}}$. By induction, $\alpha^{m}_{a, b}=\alpha_{a^{m}, b^{m}}$ for all $m\geq1$. If $a$ or $b$ is nilpotent, then $\alpha_{a, b}$ is nilpotent. So $\langle a, b\rangle=Tr(\alpha_{a, b})=0$. Hence, if $a\in \rad A$ or $b\in \rad A$, then $\langle a, b\rangle=0$. Thus the bilinear form descends to $\overline{\langle\cdot,\cdot\rangle}$ on $A/\rad A$. By the universal property of the quotient, we have the commutative diagram:
\[
\begin{tikzcd}
	{A\times A} && k \\
	{A/\rad A\times A/\rad A}
	\arrow["{\langle\cdot,\cdot\rangle}", shorten <=6pt, shorten >=6pt, from=1-1, to=1-3]
	\arrow[from=1-1, to=2-1]
	\arrow["{\overline{\langle\cdot,\cdot\rangle}}"', bend right=6pt, dashed, from=2-1, to=1-3],
\end{tikzcd}
\]
where $A/\rad A\cong kQ_{0}$. Let $a: i\to j$ be an arrow, then $a=e_{j}a e_{i}$ where $e_{i}$ (or $e_{j}$) is the trivial path at $i$ (or $j$). So $\alpha_{e_{j},e_{i}}(a)=a$. There is an equation: 
\[
\alpha_{e_{j},e_{i}}\Bigl(\sum_{p:\ \text{path}} c_{p}p\Bigr)=\sum c_{p}e_{j}pe_{i}=\sum c_{p}p,
\]
where the source of $p$ is $i$, the target is $j$, and $c_p\in k$. 

Hence, $\alpha_{e_{j},e_{i}}: A\to A$ is the projection onto $e_{j}Ae_{i}$ along $\bigoplus_{k\ne j, l\ne i}e_{k}Ae_{l}$. So,
\[
Tr(\alpha_{e_{j},e_{i}})=\dim e_{j}Ae_{i}=\dim \Hom(P_{j},P_{i}).
\]
Therefore, the matrix for $\overline{\langle\cdot,\cdot\rangle}$ is $(\dim \Hom(P_{j}, P_{i}))_{i,j}$. In our chosen basis, all non‑trivial paths belong to $\rad A$, so their pairings with any basis element vanish.
Consequently the matrix of $\langle\cdot,\cdot\rangle$ has the form
\[
\begin{pmatrix}
(\dim e_j A e_i)_{j,i=1}^n & 0 \\
0 & 0
\end{pmatrix},
\]
and the upper‑left block is exactly the Cartan matrix.
\end{proof}

\begin{remark}
$\langle\cdot,\cdot\rangle$ is symmetric if and only if the Cartan matrix of $A$ is symmetric.
\end{remark}

\begin{corollary}
If $Q$ is a finite acyclic quiver and $A=kQ$ is the path algebra, then $HH_0(\cs)=HH_{1}(\cs)=0$.
\end{corollary}
\begin{proof}
By Proposition~\ref{Cartan}, the kernel of the map $\tau$ (induced by the bilinear form $\langle\cdot,\cdot\rangle$) is the ideal generated by the arrows $kQ_1$, which coincides with the commutator ideal of $A$. From the definition of $M_*(\cs)$ we therefore obtain $HH_0(\cs)=0$. 
The long exact sequence
\[
\xymatrix{
0\ar[r] & HH_{1}(\cs) \ar[r] & HH_{0}(A) \ar[r] & DHH_{0}(A) \ar[r] & HH_{0}(\cs) \ar[r] & 0,
}
\]
together with the isomorphism $HH_{0}(A)\cong DHH_{0}(A)$, forces $HH_{1}(\cs)=0$. 
\end{proof}

\begin{theorem}\label{main}
 Let $k$ be an algebraically closed field and $A$ a finite-dimensional basic $k$-algebra.        
 The Cartan matrix of $A$ is symmetric if and only if $DM_*(\cs)\cong\Sigma ^{-1} M_*(\cs)$.  
\end{theorem}
\begin{proof}
 The equivalence follows immediately from Proposition~\ref{complex} (which links the symmetry of the form to the self‑duality of the mixed complex), Proposition~\ref{Cartan} (which identifies the form with the Cartan matrix) and Gabriel’s Structure Theorem for basic algebras.   
\end{proof}

\begin{corollary}\label{iso}
If $A$ is symmetric, then there is a canonical isomorphism
\[HH_{*}^{sg}(A)\cong HH_{*}(\cs).\]
\end{corollary}
\begin{proof}
    A symmetric algebra has a symmetric Cartan matrix. Combining this with Theorem \ref{main} and the isomorphism $DHH_n^{sg}(A, A)\cong HH_{n+1}^{sg}(A,A)$ \cite{EuSchedler09}, we conclude that the canonical isomorphism holds for all degrees.
\end{proof}

\begin{remark}
Corollary \ref{iso} does not hold for the Frobenius algebra of Example \ref{3.6}.   
The isomorphism $M_*(\cs)\cong\Sigma DM_*(\cs)$ no longer holds if $A$ is merely Frobenius (selfinjective), either. See Example \ref{ex}. 
\end{remark}

Suppose that $A$ is a symmetric algebra ($A\cong DA$ as $A^e$-modules), then:
\begin{align*}
	D\Hom(P_{i},P_{j}) &\cong \Hom(P_{j}, P_{i}\otimes_{A}DA)\\
	&\cong \Hom(P_{j}, P_{i}\otimes_{A}A)\\
	&\cong \Hom(P_{j}, P_{i}).
\end{align*}

If $A$ is Frobenius with Nakayama automorphism $\nu: A\to A$, then $A_{\nu}\cong DA$ and:
\begin{align*}
	D\Hom(P_{i},P_{j}) &\cong \Hom(P_{j}, P_{i}\otimes_{A}DA)\\
	&\cong \Hom(P_{j}, P_{i}\otimes_{A}A_{\nu})\\
	&\cong \Hom(P_{j}, (P_{i})_{\nu}).
\end{align*}
We have $(P_{i})_{\nu}=P_{\nu({i})}$ for some permutation $\nu: Q_{0}\to Q_{0}$. Thus $\langle\cdot,\cdot\rangle$ is not symmetric, but $\langle e_{i}, e_{j}\rangle=\langle e_{j},e_{\nu({i})}\rangle$ for all $i,j\in Q_{0}$.

\begin{example}\label{ex}[A Frobenius algebra with non-symmetric Cartan matrix]
Let $k$ be a field and consider the quiver $Q$ arranged as a triangle:
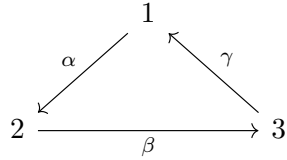
\[
\begin{tikzcd}[column sep=1.2cm, row sep=1cm]
& 1 \ar[dl, "\alpha" swap] & \\
2 \ar[rr, "\beta"'] & & 3 \ar[ul, "\gamma" swap]
\end{tikzcd}
\]
with the relation that all paths of length $2$ are zero:
\[
I = \langle \alpha\beta,\ \beta\gamma,\ \gamma\alpha \rangle \subseteq kQ.
\]
The algebra $A = kQ/I$ is $6$-dimensional and Frobenius (self-injective). Its Cartan matrix is
\[
C = \begin{pmatrix}
1 & 0 & 1 \\
1 & 1 & 0 \\
0 & 1 & 1
\end{pmatrix},
\]
which is clearly non-symmetric. 
By Theorem \ref{main}, the isomorphism $DM_*(\cs)\cong\Sigma ^{-1} M_*(\cs)$ does not hold for this algebra.
\end{example}

\section*{Acknowledgments}
The first author is deeply grateful to Professor  Bernhard Keller for many helpful discussions. The first author acknowledges support by Taiyuan Normal University
and the excellent working conditions at Institut de Math\'ematiques de Jussieu-Paris Rive
Gauche (IMJ-PRG) and Universit\'e Paris Cit\'e where part of this project was completed. This research was partially supported by the National Natural Science Foundation of China (Grant No. 12301055).

\end{document}